\documentclass[12pt,reqno,a4paper]{amsart}  

\usepackage{amsmath,amsthm,amssymb,amscd,mathdots,mathtools,enumerate,float,shuffle,stmaryrd}

\usepackage{hyperref}
\usepackage{tikz,tikz-cd}
\allowdisplaybreaks
\usetikzlibrary{decorations.pathmorphing,shapes,calc}

\newtheorem{theorem}{Theorem}[section]
\newtheorem*{theorem*}{Theorem}

\newtheorem{lemma}[theorem]{Lemma}

\numberwithin{equation}{section}

\theoremstyle{definition}

\mathtoolsset{showonlyrefs} 

\newcommand{\h}{\mathfrak H}

\newcommand{\Ha}{\mathbb{H}}
\newcommand{\Q}{\mathbb{Q}}
\newcommand{\Z}{\mathbb{Z}}

\newcommand{\hz}{\h^{2}}
\newcommand{\df}{\coloneqq}

\newcommand{\one}{{\bf 1}}

\newcommand{\qsh}{\ast_\diamond}

\newcommand{\QL}{\Q\langle L \rangle}

\newcommand{\mz}{\mathcal Z}

\definecolor{mycolor}{RGB}{194, 8, 88}
\newcommand{\todo}[1]{\message{LaTeX Warning: You did not finish your work :-( on input line \the\inputlineno} {\color{mycolor} {\big[\,}{\bf Todo:} #1\,\big]}}

\newcommand{\kk}{{\bf k}}

\makeatletter
\newcommand{\cop}{\DOTSB\cop@\slimits@}
\newcommand{\cop@}{\mathop{\bigstar}}
\newcommand{\cop@@}[2]{%
  \vphantom{\sum}%
  \ifx#1\displaystyle\big#2\else#2\fi
}
\makeatother


\title{MacMahon's sums-of-divisors and their connection to multiple Eisenstein series}

\author{Henrik Bachmann}
\address{Graduate School of Mathematics,  Nagoya University, Nagoya, Japan.}
\email{henrik.bachmann@math.nagoya-u.ac.jp}

\begin{document}

\date{\today}

\maketitle

\begin{abstract} 
We give explicit expressions for MacMahon's generalized sums-of-divisors $q$-series $A_r$ and $C_r$ by relating them to (odd) multiple Eisenstein series. Recently, these sums-of-divisors have been studied in the context of quasimodular forms, vertex algebras, $N=4$ $SU(N)$ Super-Yang-Mills theory, and the study of congruences of partitions. We relate them to a broader mathematical framework and give explicit expressions for both $q$-series in terms of Eisenstein series and their odd variants.
\end{abstract}

\section{Introduction}
In this paper, we give expressions for the $q$-series  
\begin{align*}
    A_r(q) &= \sum_{m_1>\dots>m_r>0} \frac{q^{m_1+\dots+m_r}}{(1-q^{m_1})^2\cdots (1-q^{m_r})^2}, \qquad (r\geq 1),
\end{align*}
and its `odd variant'
\begin{align*}
    C_r(q) &= \sum_{\substack{m_1>\dots>m_r>0\\ m_1,\dots,m_r \text{ odd}}}\frac{q^{m_1+\dots+m_r}}{(1-q^{m_1})^2\cdots (1-q^{m_r})^2}, \qquad (r\geq 1)
\end{align*}
in terms of Eisenstein series. These series, whose coefficients can be seen as a generalization of sums-of-divisors, were introduced by MacMahon in \cite{M}. Recently, these $q$-series have emerged in several contexts: they are linked with quasimodular forms (\cite{AR}), they appear in $N=4$ $SU(N)$ Super-Yang-Mills theory as characters of certain vertex algebras (\cite{AKM},\cite{Hu}), and congruences of their coefficients were studied in \cite{AOS}. In addition to the extensions of $A_r$ and $C_r$ we study in this work, further generalizations have been explored in \cite{AAT}. 

In \cite{AR} it was shown that $A_r$ is a quasi-modular form (of mixed weight), that is, $A_r(q) \in \Q[G_2,G_4,G_6]$. This was done by giving a recursive formula for $A_r$ in terms of $G_2 = A_1 - \frac{1}{24}$ and its derivatives. Here, the Eisenstein series of weight $k\geq 2$ are defined by 
\begin{align*}
    G_k(q) &\df  -\frac{B_k}{2k!} + \frac{1}{(k-1)!}\sum_{m,n \geq 1} n^{k-1} q^{mn},
\end{align*}
where $B_k$ denotes the $k$-th Bernoulli number.
As an odd variant, we additionally define
\begin{align*}
    G^{\bf o}_k(q) &\df G_k(q) - G_k(q^2) = \frac{1}{(k-1)!}\sum_{\substack{m,n \geq 1\\m \text{ odd}}} n^{k-1} q^{mn}.
\end{align*}
For $k\geq 4$ the Eisenstein series $G_k$ (resp. $G^{\bf o}_k$) are modular forms for $\operatorname{SL}_2(\Z)$ (resp. $\Gamma_0(2)$).
The main result of this note is the following expressions of the generating series of $A_r$ and $C_r$, which can be used to obtain explicit (non-recursive) expressions of $A_r$ (resp. $C_r$) in terms of Eisenstein series (resp. their odd variants).
\begin{theorem} \label{thm:main}
We have
\begin{align*}
1+ \sum_{r\geq 1} A_r(q)  X^{2r} &=  \frac{2}{X} \arcsin\left( \frac{X}{2} \right) \exp\left( \sum_{j \geq 1} \frac{(-1)^{j-1}}{j} G_{2j}(q)   \left( 2 \arcsin\left( \frac{X}{2} \right)\right)^{2j} \right) \,,\\
1+\sum_{r\geq 1} C_r(q)  X^{2r} &=  \exp\left( \sum_{j \geq 1} \frac{(-1)^{j-1}}{j} G^{\bf o}_{2j}(q)  \left( 2 \arcsin\left( \frac{X}{2} \right)\right)^{2j} \right) \,.
\end{align*}
\end{theorem}
To prove this theorem, we will relate the $q$-series $A_r$ and $C_r$ to (odd) multiple Eisenstein series and their Fourier expansion. In particular, this will also provide linear combinations of $A_r$, which are quasimodular of homogeneous weight. Multiple Eisenstein series can be written in terms of more general $q$-series defined for $k_1,\dots, k_r \geq 1$ by
\begin{align}\label{eq:defmonog}
g(k_1,\dots,k_r) = \sum_{\substack{m_1 > \dots > m_r > 0\\ n_1, \dots , n_r > 0}} \frac{n_1^{k_1-1}}{(k_1-1)!} \dots \frac{n_r^{k_r-1}}{(k_r-1)!}  q^{m_1 n_1 + \dots + m_r n_r } \,,\\\label{eq:defoddmonog}
g^{\bf o}(k_1,\dots,k_r) = \sum_{\substack{m_1 > \dots > m_r > 0\\ n_1, \dots , n_r > 0\\ m_1,\dots,m_r \text{ odd}}} \frac{n_1^{k_1-1}}{(k_1-1)!} \dots \frac{n_r^{k_r-1}}{(k_r-1)!}  q^{m_1 n_1 + \dots + m_r n_r } \,.
\end{align}
The $g(k)$ be seen as generalizations of the generating series of divisor sums, since in the $r=1$ case these are exactly $G_k$ without the constant term.
Notice that we have
\begin{align*}
    A_r(q) = g(\underbrace{2,\dots,2}_r) =: g(\{2\}^r)\quad \text{ and }\quad 
    C_r(q) = g^{\bf o}(\{2\}^r),
\end{align*}
which follows directly from $\sum_{n>0} n q^{mn} = \frac{q^m}{(1-q^m)^2}$. 
The $q$-series \eqref{eq:defmonog} were introduced in \cite{B2} and studied in more detail in \cite{BK}. The series in \eqref{eq:defoddmonog} are related to the level $2$ series studied in \cite{YZ}.
The $g$ are $q$-analogues of multiple zeta values, since one can show that for $k_1 \geq 2$
\begin{align*}
\lim\limits_{q\rightarrow 1}(1-q)^{k_1+\dots+k_r }g(k_1,\ldots,k_r) = \zeta(k_1,\dots,k_r),
\end{align*}
where for $k_1\geq 2, k_2,\dots,k_r\geq 1$ the multiple zeta values are defined by 
\begin{align*}
	\zeta(k_1,\ldots,k_r) \df \sum_{m_1>\cdots>m_r>0} \frac{1}{m_1^{k_1}\cdots m_r^{k_r}}.
\end{align*}
Similarly, the $g^{\bf o}$ are $q$-analogues of multiple $t$-values introduced in \cite{H2}:
\begin{align*}
\lim\limits_{q\rightarrow 1}(1-q)^{k_1+\dots+k_r }g^{\bf o}(k_1,\ldots,k_r) = t(k_1,\dots,k_r) \df \sum_{\substack{m_1>\cdots>m_r>0\\ m_1,\dots,m_r \text{ odd}}} \frac{1}{m_1^{k_1}\cdots m_r^{k_r}}.
\end{align*}
The $A_r(q)$ are $q$-analogues of $\zeta(2,\dots,2)$, which, as a consequence of \eqref{eq:zt}, can be expressed as
\begin{align}\label{eq:zeta222}
  \lim_{q\rightarrow 1} (1-q)^{2r}A_r(q) =   \zeta(\underbrace{ 2,\dots,2}_r) = \frac{\pi^{2r}}{(2r+1)!}.
\end{align}
Similar expressions also exist for $t(2,\dots,2) \in \Q \pi^{2r}$ (see \cite[Theorem 3.7]{H2}).
It should be mentioned that for any $k\geq 1$ the $q$-series $g(2k,\dots,2k)$ is quasi-modular of mixed weight (see \cite[Corollary 2.24.]{B3}). In particular, we would expect Theorem \ref{thm:main} to be generalized using the same methods as explained below. Our main theorem might be related to \cite[Theorem 1.4]{AOS} and should provide a compact way for the explicit formulas in \cite{Hu}.

\subsection*{Acknowledgements}\mbox{}\\
This project was partially supported by JSPS KAKENHI Grant 23K03030. The author would like to thank Sven M\"oller for making him aware of the appearances of $q$-analogues of multiple zeta values in the context of vertex algebras and the work \cite{AKM}.

\section{Multiple Eisenstein series}
In this section, we recall basic facts about multiple zeta values, multiple Eisenstein series, and the calculation of their Fourier expansion. Details can be found in \cite{B1},\cite{B2},\cite{B3}, and \cite{BT}. Further, we introduce the odd multiple Eisenstein series as natural odd counterparts of the classical multiple Eisenstein series.
Following \cite{GKZ}, we define for $ k_1,\dots,k_r \geq 2$ and $\tau \in \Ha$ the \emph{multiple Eisenstein series}\footnote{In the case $k_1=2$ one needs to use Eisenstein summation.} by
\begin{align}\label{eq:defmes}
\mathbb{G}_{k_1,\dots,k_r}(\tau) := \sum_{\substack{\lambda_1 \succ \dots \succ \lambda_r \succ 0\\ \lambda_i \in \Z \tau + \Z}} \frac{1}{\lambda_1^{k_1} \dots \lambda_r^{k_r}}   \,,
\end{align}
where the order $\succ$ on the lattice $\Z \tau + \Z$ is defined by $m_1 \tau + n_1 \succ m_2 \tau + n_2$ iff $m_1 > m_2$ or $m_1 = m_2 \wedge n_1 > n_2$. Since $\mathbb{G}_{k_1,\dots,k_r}(\tau + 1) = \mathbb{G}_{k_1,\dots,k_r}(\tau)$ the multiple Eisenstein series possess a Fourier expansion, i.e., an expansion in $q=e^{2\pi i \tau}$, which was calculated in \cite{GKZ} for the $r=2$ case and for arbitrary depth by the author in \cite{B2}. In depth one, we have for $k\geq 2$
\begin{align*} \mathbb{G}_k(\tau) = \sum_{\substack{\lambda \in \Z \tau + \Z\\  \lambda \succ 0}} \frac{1}{\lambda^k }  = \sum_{ \substack{ m > 0 \\ \vee \, (   m=0 \wedge n>0) }} \frac{1}{(m\tau +n)^k }  =  \zeta(k)+  \sum_{m>0} \underbrace{\sum_{n\in \Z}\frac{1}{(m\tau +n)^k}}_{{\large =: \Psi_{k}(m\tau) } }\,.
\end{align*}
Notice that these are, up to a power of $2\pi i$, the Eisenstein series in the introduction and we have $\mathbb{G}_k(\tau) = (2\pi i)^k G_k(q)$.
Here, the $\Psi_{k}(\tau)$ are sometimes referred to as \emph{ monotangent function} (\cite{Bo}). By the Lipschitz formula 
 \begin{align}\label{eq:defmonotangent}
    \Psi_k(\tau) = \sum_{n \in \Z} \frac{1}{(\tau+n)^k} = \frac{(-2\pi i)^{k}}{(k-1)!} \sum_{d>0} d^{k-1} q^d  \,, \qquad (q=e^{2 \pi i \tau})
\end{align}
 we obtain
\begin{align*}
\mathbb{G}_k(\tau)  &= \zeta(k) + \sum_{m>0}  \Psi_k(m\tau) =  \zeta(k) + \frac{(-2\pi i)^{k}}{(k-1)!}\sum_{\substack{m>0\\ d >0}} d^{k-1} q^{m d} = \zeta(k) + (-2\pi i)^k g(k) \,.
\end{align*} 
Similarly, we can define the \emph{odd multiple Eisenstein series} for $ k_1,\dots,k_r \geq 2$ and $\tau \in \Ha$ by
\begin{align*}
\mathbb{G}^{\bf o}_{k_1,\dots,k_r}(\tau) := \sum_{\substack{\lambda_1 \succ \dots \succ \lambda_r \succ 0\\ \lambda_i \in \Z^{\bf o} \tau + \Z}} \frac{1}{\lambda_1^{k_1} \dots \lambda_r^{k_r}}  \,,
\end{align*}
where the order $\succ$ is the same as before and $\Z^{\bf o} \df \{ 2m+1 \mid m\in \Z \}$ denotes the set of odd integers. This definition is new, but in the $r=2$ case they are linear combinations of the level $2$ double Eisenstein series introduced in \cite{KT}.

We get similar results as above for the depth $r=1$ case, except that we do not have a constant term, since the case $m=0$ does not appear.
In the Fourier expansion of (multiple) Eisenstein series, the $q$-series $g$ always appear together with a power of $-2\pi i$, and therefore we set for $k_1,\dots,k_r \geq 1$
\begin{align*}
     \hat{g}(k_1,\dots,k_r;\tau) =   \hat{g}(k_1,\dots,k_r) \df (-2\pi i)^{k_1+\dots + k_r} g(k_1,\dots,k_r) \,.
\end{align*}
and similarly define $\hat{g}^{\bf o}$.
With this, the higher depths analogue of $\mathbb{G}_k(\tau) = \zeta(k) + \hat{g}(k)$ and $\mathbb{G}^{\bf o}_k(\tau) = \hat{g}^{\bf o}(k)$ is given as follows.

\begin{theorem}\label{thm:mesfourier}
 For $k_1,\dots,k_r \geq 2$ and $q=e^{2\pi i \tau}$ we have the following expressions for multiple Eisenstein series and their odd variants.
\begin{enumerate}[(i)]
\item There exist integers $\alpha^{k_1,\dots,k_r}_{l_1,\dots,l_r,j} \in \Z$ with
    \begin{align*}
    \mathbb{G}_{k_1,\dots,k_r}(\tau) = \zeta(k_1,\dots,k_r) +\!\!\!\!\!\!\!\!\!\!\sum_{\substack{0 < j < r\\l_1+\dots+l_r = k_1+\dots+k_r\\l_1\geq 2,l_2,\dots,l_r\geq 1}} \!\!\!\!\!\!\!\!\!\! \alpha^{k_1,\dots,k_r}_{l_1,\dots,l_r,j}\, \zeta(l_1,\dots,l_j)  \hat{g}(l_{j+1},\dots,l_r) +  \hat{g}(k_1,\dots,k_r)\,.
\end{align*}
\item  There exist integers $\beta^{k_1,\dots,k_r}_{l_1,\dots,l_r,j} \in \Z$ with \begin{align*}
    \mathbb{G}^{\bf o}_{k_1,\dots,k_r}(\tau) = \sum_{\substack{0 < j < r\\l_1+\dots+l_r = k_1+\dots+k_r\\l_1\geq 2,l_2,\dots,l_r\geq 1}} \!\!\!\!\!\!\!\!\!\! \beta^{k_1,\dots,k_r}_{l_1,\dots,l_r,j}\, \zeta(l_1,\dots,l_j)  \hat{g}^{\bf o}(l_{j+1},\dots,l_r) +  \hat{g}^{\bf o}(k_1,\dots,k_r)\,.
\end{align*}
\end{enumerate}
\end{theorem}
Part (i) of Theorem \ref{thm:mesfourier} was first proven in \cite{GKZ} for $r=1,2$ and in \cite{B2} for general $r\geq 1$. We recall the proof in the following and show that (ii) follows with exactly the same argument. First, observe that for $k_1,\dots,k_r \geq 2$ we have, by the Lipschitz formula \eqref{eq:defmonotangent}, that the $q$-series $\hat{g}$ can be written as an ordered sum over monotangent functions
\begin{align}\label{eq:ghatclassical}
    \hat{g}(k_1,\dots,k_r) = \sum_{m_1 > \dots > m_r > 0} \Psi_{k_1}(m_1 \tau) \cdots \Psi_{k_r}(m_r \tau) \,.
\end{align}
The same expression is true for $\hat{g}^{\bf}$ by taking the sum over odd $m_j$. In what follows, we will restrict ourselves mostly to the classical case, since the odd case follows with the same argument by restricting all ordered sums over $m_j$ to the odd cases.
In general, multiple Eisenstein series can be written as ordered sums over \emph{multitangent functions} (\cite{Bo}), which are for $k_1,\dots,k_r \geq 2$ and $\tau \in \Ha$ defined by
\begin{align*}
    \Psi_{k_1,\ldots,k_r}(\tau) := \sum_{\substack{n_1>\cdots >n_r \\n_i \in \Z}} \frac{1}{(\tau+n_1)^{k_1}\cdots (\tau+n_r)^{k_r}}.
\end{align*}
These functions were originally introduced by Ecalle and then in detail studied by Bouillot in \cite{Bo}. To write $\mathbb{G}_{k_1,\dots,k_r}(\tau)$ in terms of these functions, one splits up the summation in the definition \eqref{eq:defmes} into $2^r$ parts, corresponding to the different cases where either $m_i = m_{i+1}$ or $m_i > m_{i+1}$ for $\lambda_i = m_i \tau + n_i$ and $i=1,\dots,{r}$ ($\lambda_{r+1}=0$). Then one can check that the multiple Eisenstein series can be written as 
\begin{align}\label{eq:classicalmesasmouldproduct}
    \mathbb{G}_{k_1,\dots,k_r}(\tau) = \sum_{j=0}^r \hat{g}^*(k_1,\dots,k_j) \zeta(k_{j+1},\dots,k_r)\,,
\end{align} 
where the $q$-series $\hat{g}^*$ are given as ordered sums over multitangent functions by
\begin{align}\label{eq:gastclassical}
    \hat{g}^*(k_1,\dots,k_r) := \sum_{\substack{1 \leq j \leq r\\0 = r_0< r_1 < \dots < r_{j-1} < r_j = r\\ m_1 > \dots > m_j > 0}}  \prod_{i=1}^j \Psi_{k_{r_{i-1}+1},\ldots,k_{r_i}}(m_i \tau)\,.
\end{align}
Further, one can show (\cite[Construction 6.7]{B1}) that the $q$-series $\hat{g}^*$ satisfy the harmonic product formula, e.g. $\hat{g}^*(k_1) \hat{g}^*(k_2) = \hat{g}^*(k_1,k_2) + \hat{g}^*(k_2,k_1) + \hat{g}^*(k_1+k_2)$. To obtain the statement in Theorem \ref{thm:mesfourier}, one then uses the following theorem.

\begin{theorem}{\cite[Theorem 6]{Bo}}\label{thm:reductionmonotangent}
For $k_1,\dots,k_r \geq 2$ with $k=k_1+\dots+k_r$ the multitangent function can be written as 
\begin{align*}
    \Psi_{k_1,\dots,k_r}(\tau) = \sum_{\substack{1\leq j \leq r\\ l_1+\dots+l_r = k}} (-1)^{l_1+\dots+l_{j-1}+k_j+k} \prod_{\substack{1\leq i \leq r\\i\neq j}}\binom{l_i-1}{k_i-1} \zeta(l_1,\dots,l_{j-1}) \,\Psi_{l_j}(\tau)\, \zeta(l_r,l_{r-1},\dots,l_{j+1})\ .
\end{align*} 
Moreover, the terms with $\Psi_{1}(\tau)$ vanish. 
\end{theorem}

\begin{proof}
This follows by using partial fraction decomposition 
\begin{align*}
  \frac{1}{(\tau+n_1)^{k_1}\cdots (\tau+n_r)^{k_r}} = \sum_{\substack{1\leq j \leq r\\ l_1+\dots+l_r = k}} \prod_{i=1}^{j-1} \frac{(-1)^{l_i}\binom{l_i-1}{k_i-1}}{(n_i-n_j)^{l_j}} \frac{(-1)^{k+k_j}}{(\tau+n_j)^{l_j}} \prod_{i=l+1}^r \frac{\binom{l_i-1}{k_i-1}}{(n_j-n_i)^{l_j}}\,.
\end{align*}
In order to show that the terms with  $\Psi_{1}(\tau)$ vanish, one uses the shuffle product formula for multiple zeta values and the so-called antipode relation.
\end{proof}
Applying Theorem \ref{thm:reductionmonotangent} to \eqref{eq:gastclassical}, we see by \eqref{eq:ghatclassical} that the $\hat{g}^*$ can be written as a $\mz$-linear combination of $\hat{g}$. This proves Theorem \ref{thm:mesfourier} (i) since one can also show that all the appearing multiple zeta values have the correct depth. With the same argument, also the odd case follows. Notice, however, that there are no multiple zeta values in the odd variant of \eqref{eq:classicalmesasmouldproduct}, since the $m_j$ can not be $0$. Instead, we directy get 
\begin{align}\label{eq:gastodd}
    \mathbb{G}^{\bf o}(k_1,\dots,k_r) := \sum_{\substack{1 \leq j \leq r\\0 = r_0< r_1 < \dots < r_{j-1} < r_j = r\\ m_1 > \dots > m_j > 0\\ m_1,\dots,m_j \text{odd}}}  \prod_{i=1}^j \Psi_{k_{r_{i-1}+1},\ldots,k_{r_i}}(m_i \tau)\,.
\end{align}
Theorem \ref{thm:mesfourier} (ii) then follows as before from Theorem \ref{thm:reductionmonotangent} and the odd variant of \eqref{eq:ghatclassical}.

\section{Quasi-shuffle products}
In this section, we recall some basic facts on quasi-shuffle products as, for example, presented in \cite{HI}. Let $L$ be a countable set, called \emph{alphabet}, whose elements we will refer to as \emph{letters}. A monic monomial in the non-commutative polynomial ring $\QL$ will be called a \emph{word}, and we denote the empty word by $\one$. 
Suppose that we have a commutative and associative product $\diamond$ on the vector space $\Q L$. Then the \emph{quasi-shuffle product}  $\qsh$ on $\QL$ is defined as the $\Q$-bilinear product, which satisfies $\one \qsh w = w \qsh \one = w$ for any word $w\in \QL$ and
\begin{align*}
	a w \qsh b v = a (w \qsh b v) + b (a w \qsh v) + (a \diamond b) (w \qsh  v) 
\end{align*}
for any letters $a,b \in L$ and words $w, v \in \QL$. This gives a commutative $\Q$-algebra $(\QL, \qsh)$, which is called quasi-shuffle algebra. 

In the context of multiple zeta values and multiple Eisenstein series, we are interested in the alphabet $L_z=\{z_k \mid k\geq 1\}$ together with the product $z_{k_1} \diamond z_{k_2} = z_{k_1+k_2}$ for $k_1,k_2\geq 1$. The corresponding quasi-shuffle product $\ast=\qsh$ is called the \emph{harmonic product}. 

The space $\h^1 :=\Q\langle L_z\rangle$ equipped with the harmonic product gives a commutative $\Q$-algebra $\h^1_\ast$. For an index $\kk = (k_1,\dots,k_r) \in \Z_{\geq 1}^r$ we define the word $z_\kk = z_{k_1} \cdots z_{k_r}$. Notice that, as a $\Q$-vector space, $\h^1$ is spanned by $z_\kk$ for arbitrary indices $\kk$. We define the subspace $\h^0$ as the space spanned by $z_\kk$ for admissible (meaning $k_1\geq 2$) indices $\kk$. We also consider the following subspace of $\h^0$
\begin{align*}
    \h^{2} = \Q+\langle k_1,\dots,k_r \mid r\geq 1, k_1,\dots, k_r \geq 2\rangle_\Q\,,
\end{align*}
which is spanned by $z_\kk$ such that the multiple Eisenstein series $\mathbb{G}_{\kk}$ is defined. Notice that both $\h^0$ and $\h^2$ are closed under $\ast$. We obtain the following inclusion of $\Q$-algebras
\begin{align*}
	    \hz_\ast 	\subset    \h^0_\ast &\subset \h^1_\ast \,.
\end{align*}
We can view the multiple zeta values as a $\Q$-linear map defined for $w \in \h^0$ by
\begin{align*}
    \zeta: \h^0 &\longrightarrow \mz\,,\\
   w= z_\kk &\longmapsto \zeta(w)=\zeta(\kk)\,.
\end{align*} 
By the definition of multiple zeta values as an iterated sum, it is easy to see that this map is an algebra homomorphism $\h^0_\ast\rightarrow \mz$. Similarly, one can show that the $\Q$-linear maps defined by 
\begin{minipage}{.5\linewidth}
\begin{align*}
    \mathbb{G}: \h^2 &\longrightarrow \mathcal{O}(\Ha)\,,\\
   z_\kk &\longmapsto \mathbb{G}_\kk\,.
\end{align*} 
\end{minipage}%
\begin{minipage}{.5\linewidth}
\begin{align*}
    \mathbb{G}^{{\bf o}}: \h^2 &\longrightarrow \mathcal{O}(\Ha)\,,\\
   z_\kk &\longmapsto \mathbb{G}^{{\bf o}}_\kk\,.
\end{align*} 
\end{minipage}
are also algebra homomorphism. 

For Theorem \ref{thm:main} we need the following fact for quasi-shuffle algebras: Assume that we have an algebra homomorphism $\varphi: (\QL, \qsh) \rightarrow R$ in some $\Q$-algebra $R$. Then for any $a \in L$  the following identity holds (see \cite[(2.13)]{B3} or \cite[Corollary 5.1.]{HI}) 
\begin{align*}
1+ \sum_{n=1}^{\infty} \varphi(a^n) T^n = \exp\left( \sum_{n=1}^{\infty} (-1)^{n-1} \varphi( a^{\diamond n}) \frac{T^n}{n} \right). 
\end{align*}
Here $a^{\diamond n} = a \diamond \cdots \diamond a$ and $a^n = a\cdots a$. In particular, by choosing $a=z_k$ and the harmonic product, we get for any $k\geq 2$
\begin{align}\label{eq:expG}
1+ \sum_{n=1}^{\infty} \mathbb{G}_{\underbrace{k,\dots,k}_n}(\tau) T^n = \exp\left( \sum_{n=1}^{\infty} (-1)^{n-1} \mathbb{G}_{nk}(\tau) \frac{T^n}{n} \right).
\end{align}
This shows that for any $n,k \geq1$ the multiple Eisenstein series $\mathbb{G}_{2k,\dots,2k}(\tau)$ is a (quasi)modular form of homogeneous weight $2nk$. Similarly, \eqref{eq:expG} holds for $\mathbb{G}^{{\bf o}}$ and we see that $\mathbb{G}^{{\bf o}}_{2k,\dots,2k}(\tau)$ is a polynomial in $\mathbb{G}^{{\bf o}}_{2l}$ with $1 \leq l \leq n$. 

\section{Multiple Eisenstein series \texorpdfstring{$\mathbb{G}_{2,\dots,2}$}{G2..2}, \texorpdfstring{$\mathbb{G}^{\bf o}_{2,\dots,2}$}{Go2..2} and proof of Theorem \ref{thm:main}}

Writing down \eqref{eq:classicalmesasmouldproduct} and \eqref{eq:gastclassical} for the special case $k_1=\dots=k_r=2$ gives
\begin{align*}
\mathbb{G}_{2}(\tau) &= \zeta(2) + \sum_{m>0} \Psi_2(m\tau) = \zeta(2) +  \hat{g}(2),\\
\mathbb{G}_{2,2}(\tau) &:= \zeta(2,2) + \zeta(2) \cdot \sum_{m>0} \Psi_2(m\tau) + \sum_{m>0} \Psi_{2,2}(m\tau) + \sum_{m_1 > m_2 > 0} \Psi_2(m_1 \tau) \Psi_2(m_2 \tau), \\
\mathbb{G}_{2,2,2}(\tau) &= \zeta(2,2,2) + \zeta(2,2) \cdot \sum_{m>0} \Psi_2(m\tau) + \zeta(2) \cdot  \sum_{m>0} \Psi_{2,2}(m\tau)  + \sum_{m>0} \Psi_{2,2,2}(m\tau)\\
&+\zeta(2) \sum_{m_1 > m_2 > 0} \Psi_2(m_1 \tau) \Psi_2(m_2 \tau) + \sum_{m_1 > m_2 >0} \Psi_{2,2}(m_1\tau) \Psi_2(m_2\tau)\\
&+ \sum_{m_1 > m_2 >0} \Psi_2(m_1\tau)\Psi_{2,2}(m_2\tau) + \sum_{m_1 > m_2 >m_3>0} \Psi_2(m_1\tau) \Psi_2(m_2\tau) \Psi_2(m_3\tau) \,.
\end{align*}
In general, we see that the generating series of $    \mathbb{G}_{2,\dots,2}$ can be written as
\begin{align*} 
 \sum_{l\geq 0}     \mathbb{G}_{\{2\}^l} T^{2l+1} &=\underbrace{\left(T + \zeta(2)T^3 + \zeta(2,2)T^5+ \dots \right) }_{=:\, Z(T)} \cdot \underbrace{\prod_{m>0} \left(1 + \Psi_2(m\tau) T^2 + \Psi_{2,2}(m\tau) T^4+ \dots \right) }_{=:\, M(T)},
\end{align*}
where we set $\mathbb{G}_{\{2\}^0}:=1$. From the definition of $\zeta(2,\dots,2)$ one can see that the function $Z(T)$ is given by the product formula of the sine function
\begin{align}\label{eq:zt}
 Z(T) = \sum_{j=0}^\infty \zeta( \{ 2 \}^j) T^{2j+1} = T \prod_{m=1}^\infty \left( 1 + \frac{T^2}{m^2} \right) = \frac{\sin\left( \pi i T\right)}{\pi i } \,. 
\end{align}
To calculate $M(T)$, we need the following:
\begin{lemma}
For $n\geq 1$, we have
\[ \Psi_{\{2\}^n}(x) = \frac{\pi^{2n-2} \cdot 2^{2n-1}}{(2n)!} \Psi_2(x) \,.\]
\end{lemma}
\begin{proof} By Theorem \ref{thm:reductionmonotangent} we have
\begin{align*}
\Psi_{2,\dots,2}(x) &= \sum_{e=1}^n \zeta(\underbrace{2,\dots,2}_{e-1}) \zeta(\underbrace{2,\dots,2}_{n-e}) \Psi_2(x) =  \pi^{2n-2} \sum_{e=1}^n \frac{1}{(2e-1)!} \frac{1}{(2n-2e+1)!} \Psi_2(x),
\end{align*}
where in the second equation we used \eqref{eq:zeta222}. The statement then follows from 
\[ \sum_{e=1}^n \frac{1}{(2e-1)!} \frac{1}{(2n-2e+1)!}   = \frac{2^{2n-1}}{(2n)!}. \]
\end{proof}
Using this we get for $M(T)$
\begin{align*}
M(T) &= \prod_{m>0} \left( \sum_{j=0}^\infty \Psi_{\{2\}^j}(m\tau) T^{2j}  \right)=  \prod_{m>0} \left( 1 + \sum_{j=1}^\infty \frac{\pi^{2j-2} \cdot 2^{2j-1}}{(2j)!} T^{2j} \Psi_2(m \tau)  \right) \\
&=  \prod_{m>0} \left( 1 + \frac{1}{2\pi^2} \left( \cos(2\pi i T) -1 \right) \Psi_2(m \tau)  \right) = \prod_{m>0} \left( 1 - \frac{\sin(\pi i T)^2}{\pi^2} \Psi_2(m \tau)  \right) \\
&=  \prod_{m>0} \left( 1 + Z(T)^2 \Psi_2(m \tau)  \right),
\end{align*}
where, we used $\cos(x) - 1 = - 2\sin\left( \frac{x}{2} \right)^2$ in the fourth equation. As we have seen before, the multiple divisor functions can be written as a sum over products of monotangents functions, and we have
\[ \hat{g}(\{2\}^l) = \sum_{m_1 > \dots > m_n>0} \Psi_2(m_1\tau) \cdots \Psi_2(m_l \tau).\]
With this, we can write the generating series of $\mathbb{G}_{\{2\}^l}$ as
\begin{equation} \label{geng22}
\sum_{l\geq 0} \mathbb{G}_{\{2\}^l} T^{2l+1} = Z(T)  \prod_{m>0} \left( 1 + Z(T)^2 \Psi_2(m \tau)  \right)  = \sum_{l\geq 0}  \hat{g}(\{2\}^l) Z(T)^{2l+1}.
\end{equation}
On the other hand, by \eqref{eq:expG} we have
\[ \sum_{l\geq 0} \mathbb{G}_{\{2\}^l} T^{2l+1} = T\exp\left( \sum_{k>0} \frac{(-1)^{k-1}}{k} \mathbb{G}_{2k} T^{2k} \right) \,. \]
Theorem \ref{thm:main} (i) follows by setting $X=Z(T)$ and dividing out an appropriate power of $2\pi i$. Given that \(\mathbb{G}_{\{2\}^l}\) is a quasimodular form of homogeneous weight \(2l\), equation \eqref{geng22} can be interpreted as a methodology for combining the \(q\)-series \(g(2,\dots,2) = A_r\) to construct a quasimodular form of homogeneous weight. To get the formula in Theorem \ref{thm:main} (ii), we proceed in a similar way, by setting $k_1=\dots=k_r=2$ in \eqref{eq:gastodd}. This leads to
\begin{align*} 
 \sum_{l\geq 0}     \mathbb{G}^{\bf o}_{\{2\}^l} T^{2l} &=\prod_{\substack{m>0\\ m \text{ odd}}} \left(1 + \Psi_2(m\tau) T^2 + \Psi_{2,2}(m\tau) T^4+ \dots \right)= \sum_{l\geq 0}  \hat{g}^{\bf o}(\{2\}^l) Z(T)^{2l},
\end{align*}
which then directly gives the formula in Theorem \ref{thm:main} (ii) after dividing out some powers of $2\pi i$ and by using 
\[ \sum_{l\geq 0} \mathbb{G}^{\bf o}_{\{2\}^l} T^{2l} = \exp\left( \sum_{k>0} \frac{(-1)^{k-1}}{k} \mathbb{G}^{\bf o}_{2k} T^{2k} \right) \,. \]


\end{document}